%% file: thesis.tex
\documentclass[MSc, beforeDefense]{misc/iitthesis}

\include{misc/thesis-fields}

\include{front/personal-acks}

\include{front/abstract}
\include{front/abbrevs}

\usepackage{misc/iitthesis-extra}

\include{misc/my-general}

\include{misc/my-thesis-specific}

\usepackage{multibib}
\newcites{pubinfo}{Acknowledgement page references}
\def\iitthesisextramultibibdefs{}

\begin{document}


\makefrontmatter

%
%
\include{main/intro}

%
%
%

\include{main/prelims}

\include{main/mainchap1}
\include{main/conclusion}

%
%


%
%
\makebackmatter 


\end{document}

%% file: misc/thesis-fields.tex
\authorEnglish{Mikhail Mironov} 

\titleEnglish{Hilbert function spaces and multiplier algebras of analytic discs}

\disciplineEnglish{Mathematics}

\supervisionEnglish{The research thesis was done in the Faculty of Mathematics under the supervision of Prof.~Orr Moshe Shalit.}

\GregorianDateEnglish{August 2024}
\JewishDateEnglish{\\ Av 5784}

\financialAcknowledgementEnglish{The generous financial help of Technion's graduate school scholarship is gratefully acknowledged.} 

\publicationInfoEnglish{%
The author of this thesis states that the research, including the collection, processing and presentation of data, addressing and comparing to previous research, etc., was done entirely in an honest way, as expected from scientific research that is conducted according to the ethical standards of the academic world. Also, reporting the research and its results in this thesis was done in an honest and complete manner, according to the same standards.

}

\thesisbibfiles{back/general}
\thesisbibstyle{alpha}

%% file: front/personal-acks.tex
\personalAcknowledgementEnglish{
I would like to thank my advisor, my parents, my friends, etc. etc.

Add any thank-yous, acknowledgements, personal comments you wish to make here (in \texttt{personal-acks.tex}).

Note that this acknowledgements section only gets printed in the post-exam version of the thesis (i.e. if you leave out the \texttt{beforeDefense} option to your document class in \texttt{thesis.tex}.)
}

%% file: front/abstract.tex
\abstractEnglish{

The thesis is devoted to two related problems.
\begin{enumerate}
    \item The isomorphism problem for analytic discs: 

    Suppose $V$ is the unit disc $\mathbb{D}$ embedded in the $d$-dimensional unit ball $\mathbb{B}_d$ and attached to the unit sphere. Consider the space $\mathcal{H}_V$, the restriction of the Drury-Arveson space to the variety $V$, and its multiplier algebra $\M_V = \Mult(\H_V)$.
    The isomorphism problem is the following: Is $V_1 \cong V_2$ equivalent to $\M_{V_1} \cong \M_{V_2}$?

    A theorem of Alpay, Putinar and Vinnikov states that for $V$ without self-crossings on the boundary $\M_V$ is the space of bounded analytic functions on $V$. We consider what happens when there are self-crossings on the boundary and prove that if $\M_{V_1} \cong \M_{V_2}$ algebraically, then $V_1$ and $V_2$ must have the same self-crossings up to a unit disc automorphism. We prove that an isomorphism between $\M_{V_1}$ and $\M_{V_2}$ can only be given by a composition with a map from $V_1$ to $V_2$. In the case of a single self-crossing we show that there are only two possible candidates for this map and find these candidates. 
    \item The embedding dimension for complete Pick spaces: 

    A Theorem of Agler and McCarthy states that any complete Pick space can be realized as $\mathcal{H}_V$, for some $V$ in $\mathbb{B}_d$, where $d$ can be infinite. The smallest such $d$ is called the embedding dimension. Given a complete Pick space can we find its embedding dimension? Can we at least determine if it is finite or infinite?
    
    We look into this problem for rotation-invariant spaces on the unit disc $\D$. We prove a general result which explicitly relates the embedding dimension with the kernel of the space. This allows us to prove that the embedding dimension for certain weighted Hardy-type spaces is infinite. 
\end{enumerate}

}

%% file: front/abbrevs.tex
\abbreviationsAndNotation{
\begin{tabular}{p{2cm}@{:\quad}l}
$\N$ & the set of positive integers \\
$\R$ & the set of real numbers \\
$\C$ & the set of complex numbers \\
$\D$ & the open unit disc at $0$ in $\C$ \\
$\overline{\D}$ & the closed unit disc at $0$ in $\C$ \\
$\TT$ & the unit circle at $0$ in $\C$ \\
$\B_d$ & the open unit ball at $0$ in $\C^{d}$ \\
$\Cl X$ & the closure of $X$ \\
$\Span X$ & the linear span of $X$ \\
$\id_X$ & the identity map on $X$ \\
$I_{\H}$ & the identity operator on a Hilbert space $\H$ \\
$\Mult(\H)$ & the multiplier algebra of an RKHS $\H$ \\
$H^2(\D)$ & the Hardy space on the unit disc \\ 
$H^{\infty}(V)$ & the space of bounded analytic functions on $V$ \\ 
$H^2_d$ & the Drury-Arveson space on $\B_d$ \\
$\M_d$ & the multiplier algebra $\Mult(H^2_d)$ of the Drury-Arveson space  \\
$\H_V$ & the restriction of $H^2_d$ to the variety $V$ \\
$\M_V$ & the subalgebra of $\M_d$ of multipliers of $\H_V$ \\
$\H_f$ & the pullback of $\H_V, \, V = f(\D)$ to $\D$ \\ 
$\M_f$ & the pullback of $\M_V, \, V = f(\D)$ to $\D$ \\ 
$\k^f$ & the kernel of the RKHS $\H_f$ \\
$A_f(\xi)$ & $\langle f(\xi), f'(\xi) \xi \rangle$ for a function $f$ and $\xi \in \TT$ \\

\end{tabular}
}

%% file: misc/my-general.tex
\usepackage{graphicx}

\usepackage{xcolor}                 

 \theoremstyle{ams-remark}
 \theoremheaderfont{\normalfont\bfseries}
 \theorembodyfont{\itshape}
 \theoremseparator{.}
 \theoremnumbering{Alph} 
 \theoremsymbol{}

\newtheorem{thrmm}{Theorem}

\renewcommand{\le}{\leqslant}
\renewcommand{\ge}{\geqslant}

\let\intt\int
\renewcommand{\int}{\intt\limits}

\newcommand{\C}{\mathbb{C}} 						
\newcommand{\D}{\mathbb{D}} 						
\newcommand{\R}{\mathbb{R}}		                    
\newcommand{\N}{\mathbb{N}}	                	    
\newcommand{\TT}{\mathbb{T}}	 					
\renewcommand{\H}{\mathcal{H}}                      
\newcommand{\M}{\mathcal{M}}                        
\newcommand{\B}{\mathbb{B}}                         

\renewcommand{\a}{\alpha}									
\renewcommand{\b}{\beta}									
\newcommand{\de}{\delta}									
\newcommand{\e}{\varepsilon}								
\newcommand{\z}{\zeta}										
\renewcommand{\k}{\kappa}									
\renewcommand{\l}{\lambda}									
\let\originalnu\nu											
\renewcommand{\nu}{\originalnu} 					        
\newcommand{\f}{\varphi}									

\DeclareMathOperator{\Cl}{Cl}						    

\renewcommand{\Re}{\operatorname{Re}} 		    	

\DeclareMathOperator{\Span}{span}							
\DeclareMathOperator{\id}{id}								

\DeclareMathOperator{\Mult}{Mult}

%% file: main/intro.tex
\chapter{Introduction}
\label{chap:intro}

    \section{The isomorphism problem}
        The main problem we are interested in is the isomorphism problem for multiplier algebras of varieties in the unit ball. In this section we state this problem in the general setting.
    
        Let $\B_d, \ d \in \N \cup \{ \infty \}$ denote the Euclidean open unit ball in $\C^d$, where $\C^d$ means $\ell^2(\N)$ for $d = \infty$. The main function space on $\B_d$ we consider is the Drury-Arveson space.
        \begin{definition}
            \emph{The Drury-Arveson space on $\B_d$} is
            \begin{equation*}
                H^2_d = \left\{ f(z) = \sum_{\a \in \N_0^d} c_{\a} z^{\a}: \: ||f||^2 = \sum_{\a \in \N_0^d} |c_{\a}|^2 \frac{\a!}{|\a|!} < \infty \right\}.
            \end{equation*}
            It is a reproducing kernel Hilbert space with kernel
            \begin{equation*}
                \k(z, w) = \frac{1}{1 - \langle z, w \rangle_{\C^d}}.
            \end{equation*}
        \end{definition}
        This space is one of the multivariate generalizations of the Hardy space $H^2(\D)$ on the unit disc, i.e., $H^2(\D) = H^2_1$.
        The Drury-Arveson space naturally arises in operator theory as well as in theory of complete Pick spaces, see the survey \cite{shalit2015operator} or the more recent \cite{hartz2023invitation} for an introduction.
        
        We denote by $\M_d = \Mult (H^2_d)$ the multiplier algebra of $H^2_d$. We say that $V \subset \B_d$ is a (multiplier) \emph{variety} if it is a joint zero set of functions from $\M_d$, i.e., there exists $E \subset \M_d$ such that
        \begin{equation*}
            V = \{ z \in \B_d: \f(z) = 0, \text{ for all } \f \in E \}.
        \end{equation*}
        Note that for $d < \infty$ we can replace $E$ by a finite set (even of cardinality at most $d$), see the argument in \cite[Chapter 5, Section 7]{chirka1989complex}. In particular, such $V$ is an analytic variety in $\B_d$.
        
        
        For a variety $V \subset \B_d$ the associated Hilbert space of functions on $V$ is $\H_V \subset H^2_d$:
        \begin{equation*}
            \H_V = \Cl \Span \{ \k_{\l}: \l \in V \} = \{ f \in H^2_d: \: f|_V = 0 \}^{\perp}.
        \end{equation*}        
        This Hilbert space is a reproducing kernel Hilbert space of functions on $V$. 
        
        Next, we define the algebra associated to $V$:
        \begin{equation*}
            \M_V = \{ \f|_V: \f \in \M_d \}.
        \end{equation*}
        By \cite[Proposition 2.6]{DSR15}, $\M_V$ is exactly the multiplier algebra $\Mult(\H_V)$, and $\M_V$ is completely isometrically isomorphic to the quotient $\M_d / J_V$, where
        \begin{equation*}
            J_V = \{ \f \in \M_d: \f|_V = 0 \}.
        \end{equation*}
        Moreover, $\M_V$ is contained in $H^{\infty}(V)$, the algebra of bounded analytic functions on $V$, but might not coincide with it.
    
        \textbf{The isomorphism problem:} When is $\M_V \cong \M_W$ equivalent to $V \cong W$? 
        \\
        The answer heavily depends on the notion of isomorphism for both varieties and algebras. 
        We recall a few possible notions. First, $V$ is said to be \emph{an automorphic image} of $W$ in $\B_d$ if there exists $\mu$, an automorphism of $\B_d$, such that $V = \mu(W)$. Since $\mu^{-1}$ is also an automorphism of $\B_d$ this notion is symmetric with respect to $V$ and $W$. We say that $V$ and $W$ are \emph{biholomorphic} if there are holomorphic in $\B_d$ maps $F$ and $G$ such that $F|_V$ is a bijecton from $V$ to $W$ and $G|_W$ is its inverse. Finally, we say that $V$ and $W$ in $\B_d$ are \emph{multiplier biholomorphic} if they are biholomorphic with respect to the maps $F, G$ which are coordinate multipliers, i.e.,
        $
            F, G \in \underbrace{\M_d \times \ldots \times \M_d}_{d}.     
        $

        We refer to the survey by Salomon and Shalit \cite{Sal} for an overview of the available answers to the isomorphism problem:
        \begin{itemize}
            \item $\M_V \cong \M_W$ isometrically if and only if $\H_V \cong \H_W$ as reproducing kernel Hilbert spaces if and only if $V$ is an automorphic image of $W$, where $d < \infty$ or $V$ and $W$ have the same affine codimension, \cite[Proposition 4.8, Theorem 4.6]{Sal}, \cite{DSR15}. 
            
            Similarly, Rochberg, \cite[Theorem 7]{rochberg2019complex}, has shown that for ordered finite sets $V, W$, $\H_V \cong \H_W$ if and only if $V$ is an automorphic image of $W$, adding some concrete quantitative conditions for an isomorphism to exist, such as coinciding ``normal forms'' or having the same Gram matrix. Rochberg also showed that $V \cong W$ if and only if all triples in $V$ are congruent to all triples in $W$.
            \item $\M_V \cong \M_W$ algebraically is equivalent to $V$ and $W$ being multiplier biholomorphic for homogeneous varieties $V$ and $W$ if $d < \infty$, \cite[Theorem 5.14]{Sal}, \cite{DRS11}, \cite{HarTop}.
            \item $\M_V \cong \M_W$ algebraically implies that $V$ and $W$ are multiplier biholomorphic, for $V$ and $W$ which are irreducible varieties (or finite unions of irreducible varieties and a discrete variety), \cite[Theorem 5.5]{Sal}, \cite{DSR15}.
            \item $V \cong W$ via a biholomorphism that extends to be a 1-to-1 $C^2$-map on the boundary implies $\M_V \cong \M_W$ algebraically, for $V$ and $W$ --- images of finite Riemann surfaces under a holomap that extends to be a 1-to-1 $C^2$-map on the boundary, \cite[Corollary 5.18]{Sal}. This is the result of an idea from \cite{Vin} being generalized in \cite[Section 2.3.6]{ARS08} and culminating in \cite{Kerr}.
            \item For $V \subset \B_{\infty}$ of the form $V = f(\D)$ with
            \begin{equation*}
                f(z) = (b_1 z, b_2 z^2, b_3 z^3, \ldots), \quad \text{ where } b_1 \ne 0 \text{ and } ||b||_2^2 = 1 
            \end{equation*}
            $\M_V = H^{\infty}(V) \cong H^{\infty}(\D)$ if and only if $\sum n |b_n|^2 < \infty$, while any two such varieties are multiplier biholomorphic, \cite[Corollary 7.4]{Dav}.
        \end{itemize}
         It is still an open question whether $V \cong W$ via a multiplier biholomorphism implies $\M_V \cong \M_W$ algebraically for sufficiently simple $V$ and $W$. Though, it is known to be false for $V$, $W$ --- discrete varieties, see \cite[Example 5.7]{Sal}.

    \section{Analytic Discs}     
    
        In this paper we study the isomorphism problem in the case where $V$ and $W$ are analytic discs attached to the unit sphere. 
        \begin{definition} \label{def: analytic disc}
            \emph{An analytic disc attached to the unit sphere} is a variety $V \subset \B_d, \, d < \infty$ for which there exists an injective analytic map $f: \D \to \B_d$ with $f'(z) \ne 0, \, z \in \D$ such that $V = f(\D)$, $f$ extends to $C^2$ up to $\overline{\D}$ and 
            \begin{equation*}
                ||f(x)|| = 1 \iff |x| = 1.  
            \end{equation*}            
            We say that $f$ is \emph{an embedding map} of $V$.
        \end{definition}
        Note that by \cite[Corollary 3.2]{Dav} such discs meet the boundary transversally, i.e., 
        \begin{equation*}
            \langle f(\xi), f'(\xi) \rangle \ne 0, \quad \xi \in \TT.     
        \end{equation*}
        We emphasize that $V$ is defined to be a variety. It is not clear whether for an arbitrary $f$ satisfying properties from Definition \ref{def: analytic disc} the image $V = f(\D)$ is a variety.
        
        Instead of working directly with the spaces $\H_V$ and $\M_V$, it is more convenient to pull back this spaces from the variety $V$ to the unit disc $\D$.
        We denote by $\H_f$ an RKHS on $\D$ that we get from $\H_V$ by composing with $f$:
        \begin{equation*}
            \H_f = \{ h:\D \to \C: \: h \circ f^{-1} \in \H_V \}, \qquad ||h||_{\H_f} = || h \circ f^{-1} ||_{\H_V}.
        \end{equation*}
        The reproducing kernel of this space is
        \begin{equation} \label{eq: k^f formula}
            \k^f(z, w) = \frac{1}{1 - \langle f(z), f(w) \rangle}.
        \end{equation}
        We denote the multiplier algebra of this space by $\M_f$, so that
        \begin{equation*}
            \M_f = \{ \f:\D \to \C: \: \f \circ f^{-1} \in \M_V \}, \qquad ||\f||_{\M_f} = || \f \circ f^{-1} ||_{\M_V}.
        \end{equation*}
    
        The isomorphism problem was partially solved in the case when $f$ extends injectively to $\overline{\D}$. We state the results \cite[Proposition 2.2, Theorem 2.3]{Vin} combined with the above-mentioned fact that the transversality condition is satisfied automatically.
        \begin{thrmm} \label{thrmm: injective embedding}
            Suppose $V$ is an analytic disc attached to the unit sphere, $f$ is the embedding map of $V$. If the extension of $f$ to $\overline{\D}$ is injective, then 
            \begin{itemize}
                \item $\H_f = H^2(\D)$ with equivalent norms,
                \item $\M_f = H^{\infty}(\D)$ with equivalent norms.
            \end{itemize}
        \end{thrmm}
    
        We now state our results.
        \begin{theorem} \label{thrm: algebras to crossings}
            Suppose $V$, $W$ are analytic discs attached to the unit sphere, and $f$, $g$ are the respective embedding maps. If $\M_f = \M_g$, then
            \begin{equation*}
                f(\xi) = f(\zeta) \iff g(\xi) = g(\zeta), \quad \xi, \zeta \in \TT. 
            \end{equation*}
            This means that $f$ and $g$ have the same self-crossings on the boundary.
        \end{theorem}
        \begin{theorem} \label{thrm: isomorphic algebras to automorphism}
            Suppose $V$, $W$ are analytic discs attached to the unit sphere, and $f$, $g$ are the respective embedding maps. If $\M_f \cong \M_g$ algebraically, then there exists $\mu$, an automorphism of the unit disc $\D$, such that $\M_f = \M_{g \circ \mu}$ with equivalent norms.
        \end{theorem}
        Combining these two theorems we get: 
        \begin{theorem} \label{thrm: algebras to self-crossings}
            Suppose $V$, $W$ are analytic discs attached to the unit sphere, and $f$, $g$ are the respective embedding maps. If $\M_f \cong \M_g$ algebraically, then, up to a unit disc automorphism, $f$ and $g$ have the same self-crossings on the boundary.
        \end{theorem}
        If $f$ and $g$ have the same self-crossings on the boundary up to a unit disc automorphism we say that they have the same \emph{self-crossing type}.
        
        Theorem \ref{thrm: algebras to self-crossings} means that for $f$ and $g$ with different self-crossings types the multiplier algebras are always non-isomorphic. Hence, we found a coarse characteristic which separates non-isomorphic cases, so that it remains to solve the isomorphism problem only for embeddings with the same self-crossings type. The author does not know if the same self-crossings type implies the isomorphism of the multiplier algebras. 
        
        Note that now we have a more complete solution in the injective case:
        Theorem \ref{thrm: algebras to self-crossings} together with Theorem \ref{thrmm: injective embedding} give us
        \begin{theorem}
            Suppose $V$, $W$ are analytic discs attached to the unit sphere, and $f$, $g$ are the respective embedding maps and $f$ is injective up to $\overline{\D}$. Then $\M_f \cong \M_g$ algebraically if and only if $g$ is injective up to $\overline{\D}$.  
        \end{theorem}

        Next, we consider analytic discs with exactly one self-crossing on the boundary. Since we can always apply an automorphism we assume that $f(-1) = f(1)$ is the self-crossing and otherwise $f$ is injective in $\overline{\D}$.
        It turns out that for such maps the isomorphism condition is rigid.
        \begin{theorem} \label{thrm: multiplier isomorphism two points to two auto}
            Suppose $V$, $W$ are analytic discs attached to the unit sphere, and $f$, $g$ are the respective embedding maps with the only self-crossing at $\pm 1$. Then $\M_f \cong \M_g$ algebraically if and only if $\M_f = \M_{g \circ \mu}$, where
            \begin{equation*}
                \mu(z) = \frac{z - \a}{1 - \a z}, \text{ or} \qquad \mu(z) = \frac{\b - z}{1 - \b z}
            \end{equation*}
            with constants $\a, \b \in (-1, 1)$ that can be explicitly calculated in terms of $f(\pm 1), f'(\pm 1)$ and $g(\pm 1), g'(\pm 1)$.
        \end{theorem}   

        Let us consider an example. Set
        \begin{equation*}
            b_r(z) = \frac{z - r}{1 - rz}, \quad -1 < r < 1,
        \end{equation*}
        an automorphism  of the unit disc such that $b_r(\pm 1) = \pm 1$. Define
        \begin{equation*}
            f_r(z) = \frac{1}{\sqrt{2}} \left( z^2, b_r(z)^2 \right), \quad r > 0.
        \end{equation*}
        Then, according to \cite[Theorem 5.2]{Dav}, $V = f(\D)$ is a variety with the only self-crossing at $\pm 1$. Using Theorem \ref{thrm: multiplier isomorphism two points to two auto} we show:
        \begin{theorem} \label{thrm: f_r and f_s}
            $\M_{f_r} \ne \M_{f_s}$ for $r \ne s$.
        \end{theorem}
        We see that, in contrast to the injective case, even for embeddings with the same self-crossings the multiplier algebras might not coincide. Though, it does not mean that the algebras are not isomorphic, as we still have the freedom to do unit disc automorphisms. Let us take care of this by considering a more general version of this embedding. Set
        \begin{equation*}
            f_{r, s}(z) = \frac{1}{\sqrt{2}} \left( b_r(z)^2, b_s(z)^2 \right), \quad r \ne s,
        \end{equation*}
        so that $f_r = f_{0, r}$.
        Note that
        \begin{equation*}
            (b_r \circ b_s)(z) = \frac{z - \frac{r + s}{1 + rs}}{1 - \frac{r + s}{1 + rs} z} = b_{\frac{r + s}{1 + rs}}(z).
        \end{equation*}
        Hence,
        \begin{equation*}
            (f_{r, s} \circ b_{-r})(z) = \frac{1}{\sqrt{2}} \left( z^2, b_{\frac{s - r}{1 - sr}}(z)^2 \right) = f_{0, \frac{s - r}{1 - sr}}(z),
        \end{equation*}
        meaning that we get the same varieties (do $z \mapsto -z$ in case $s < r$). Now, for $f_{0, r}, \, r \in (0,1)$ consider 
        \begin{equation*}
            t = \frac{-1 + \sqrt{1 - r^2}}{r},
        \end{equation*}
        we have $t \in (-1, 0)$ and $f_{0, r} \circ b_t = f_{t, -t}$. In fact, $t$ is unique, so that any $f_{r, s}$ corresponds to a unique $f_{t, -t}, \, t \in (-1, 0)$ via a unit disc automorphism. 
        For such embeddings Theorem \ref{thrm: multiplier isomorphism two points to two auto} takes a simpler form
        \begin{theorem} \label{thrm: f_r-r and f_s-s}
            $\M_{f_{r, -r}} \cong \M_{f_{s, -s}}$ algebraically if and only if  $\M_{f_{r, -r}} = \M_{f_{s, -s}}$.
        \end{theorem}
        The author does not know whether $\M_{f_{r, -r}} = \M_{f_{s, -s}}$ holds for any $r$ and $s$, but it seems plausible. If there are $r$ and $s$ for which $\M_{f_{r, -r}} \ne \M_{f_{s, -s}}$, then Theorem \ref{thrm: f_r-r and f_s-s} gives us an example where two embeddings with the same self-crossing type give rise to non-isomorphic multiplier algebras.

    \section{Complete Pick spaces}
        A different way to look at the multiplier algebras $\M_V$ comes from function theory, in particular from the theory of complete Pick spaces. 
        
        Let $\H$ be a reproducing kernel Hilbert space on a set $X$ with kernel $\k$ and let $\M$ be its multiplier algebra $\Mult(\H)$. Consider the following problem. Given $\l_1, \ldots, \l_n \in X$ and $a_1, \ldots, a_n \in \C$, does the interpolation problem
        \begin{equation*}
            \f(\l_j) = a_j, \quad j = 1, \ldots, n
        \end{equation*}
        have a multiplier solution $\f \in \M, \ ||\f||_{\M} \le 1$?
        
        This problem is called \emph{the Pick interpolation problem}\/. Denote by $M_{\f}$ the operator on $\H$ given by multiplying by $\f$. By definition the norm on $\M$ is inherited from the operator norm, hence, $||\f||_{\M} \le 1$ is equivalent to $I_{\H} - M_{\f} M^*_{\f} \ge 0$. Using this fact together with $M^*_{\f} \k_x = \overline{\f(x)} \k_x$ we get a necessary condition for the Pick interpolation problem to have a solution, i.e., 
        \begin{equation} \label{eq: Pick positive matrix}
            \left[ \left( 1 - a_j \overline{a_k} \right) \k(\l_j, \l_k) \right]_{j, k = 1}^n \ge 0.
        \end{equation}
        Pick, \cite{pick1915}, showed that for the Hardy space $H^2(\D)$, which has the Szego kernel 
        \begin{equation*}
            \k(z, w) = \frac{1}{1 - z \bar w},
        \end{equation*}
        this positivity condition is also sufficient. Spaces and their kernels for which the positivity of \eqref{eq: Pick positive matrix} is sufficient for the Pick interpolation problem to have a solution are called \emph{Pick spaces} and \emph{Pick kernels} respectively. It is not difficult to see that not all spaces are Pick, e.g., the Bergman space on the unit disc, namely, the space with the kernel
        \begin{equation*}
            \k(z, w) = \frac{1}{(1 - z \bar w)^2}
        \end{equation*}
        is not Pick.
        
        Let us now turn our attention to the matrix version of the Pick interpolation problem. We consider $\M^{m \times m}$, the algebra of $m \times m$ multiplier matrices, so that $\M^{m \times m} = \Mult(\H^m)$. Then, the $m \times m$ Pick interpolation problem is the following: Given $\l_1, \ldots, \l_n \in X$ and $A_1, \ldots, A_n \in M_m(\C)$, does the interpolation problem
        \begin{equation*}
            F(\l_j) = A_j, \quad j = 1, \ldots, n
        \end{equation*}
        have an $m \times m$ multiplier solution $F \in \M^{m \times m}, \ ||F||_{\M^{m \times m}} \le 1$?
        
        Similarly to the $m = 1$ case, it is possible to show the necessity of the positivity condition
        \begin{equation*}
            \left[ \left( 1 - A_j A_k^* \right) \k(\l_j, \l_k) \right]_{j, k = 1}^n \ge 0.
        \end{equation*}
        If this condition is sufficient for all $m \ge 1$ we call the space and the respective kernel \emph{complete Pick}.  
        We refer to the monograph of Agler and McCarthy, \cite{AglBook}, for an in-depth introduction to the study of (complete) Pick kernels, and to \cite[Chapter 5, Section 3]{quiggin1994generalisations} for an example of a space which is Pick but not complete Pick. 

        We say that an RKHS $\H$ on $X$ with kernel $\k$ is \emph{irreducible} if $\k(x, y)$ does not vanish for $x, y \in X$.
        The connection between complete Pick spaces and our isomorphism problem is given by the following theorem due to Agler and McCarthy, \cite[Theorem 4.2]{Agl}, \cite[Theorem 8.2]{AglBook}. 
        \begin{thrmm} \label{thrmm: Agler-McCarthy}
            Let $\H$ be a separable irreducible reproducing kernel Hilbert space on $X$. Then $\H$ is complete Pick if and only if for some $d \in \N \cup \{ \infty \}$ there is a
            map $b: X \to \B_d$ and a nowhere vanishing function $\de: X \to \C$ such that
            \begin{equation*}
                \k(z, w) = \frac{\de(z) \overline{\de(w)}}{1 - \langle b(z), b(w) \rangle_{\C^d}}.
            \end{equation*}
            Meaning that up to rescaling and composition $\H$ is essentially $\H_V$ for some variety $V \subset \B_d$. Moreover, $\M$ is isometrically isomorphic to $\M_V$ via the same composition.
        \end{thrmm}
        
        We see that the Drury-Arveson space is, in a sense, the universal complete Pick space. Hence, the isomorphism problem is closely related to the study of irreducible complete Pick spaces and their multiplier algebras.
        
    \section{Embedding dimension}
        A natural question one might ask regarding Theorem \ref{thrmm: Agler-McCarthy} is what is the smallest $d$ for a given space $\H$. Such a $d$ is called \emph{the embedding dimension} of $\H$. The main problem is to determine this embedding dimension for a fixed complete Pick space $\H$ or at least determine if it is finite or infinite, since for finite $d$ the Drury-Arveson space and its multiplier algebra are a lot more tractable.
        
        We can ask a similar question about $\M$. Let us say that $d$ is \emph{the multiplier embedding dimension} if it is the smallest $d$ such that there is a variety $V \subset \B_d$ with $\M$ being isometrically isomorphic to $\M_V$ via a composition map. It is clear that the multiplier embedding dimension can only be smaller than the embedding dimension. But, in fact, they are equal in light of \cite[Corollary 3.2]{hartz2017isomorphism}, which implies that for complete Pick spaces coinciding multiplier algebras mean that Hilbert spaces are rescalings of each other.

        Let us list some results regarding the embedding dimension problem.
        Rochberg, \cite{Rochberg2017301}, showed that the embedding dimension of the Dirichlet space is infinite. Hartz, \cite{Har}, considered an even weaker notion of embedding and proved that even non-isometric multiplier embedding dimension for the Dirichlet space is infinite. 

        We consider the embedding dimension problem for a class of relatively simple spaces in one variable, namely, rotation-invariant spaces on the unit disc $\D$.
        \begin{definition}
            We say that an RKHS $\H$ is \emph{rotation-invariant} if for any $f \in \H$ and $\xi \in \TT$, $f(\xi z) \in \H$ and $|| f(\xi z) ||_{\H} = || f(z) ||_{\H}$.    
        \end{definition}
        This problem was considered in more generality (unitarily invariant spaces on the unit ball) by Hartz \cite{hartz2017isomorphism}.
        
        We provide a way to find the embedding dimension for such spaces.
        \begin{theorem} \label{thrm: isometric embedding}
            Let $\H$ be a rotation-invariant irreducible complete Pick space of analytic functions on the unit disc $\D$ with kernel $\k$.
            Then the embedding dimension $d$ is finite if and only if 
            \begin{equation*}
                \frac{1}{\k(0, 0)} - \frac{1}{\k(z, w)}
            \end{equation*}
            is a polynomial in $z \bar w$. Moreover, $d$ is exactly the number of non-zero coefficients of this polynomial.
        \end{theorem}
        This theorem is essentially the one-dimensional case of \cite[Proposition 11.8]{hartz2017isomorphism}.

        If one relaxes the embedding conditions so that the isomorphism given by the composition is not necessarily isometric, then, unlike in Theorem \ref{thrm: isometric embedding}, it is enough for $1/\k(0,0) - 1/\k$ to be rational to have finite (even $d = 1$) embedding dimension. 
        \begin{theorem} \label{thrm: rational q}
            Let $\H$ be an irreducible complete Pick space of analytic functions on the unit disc $\D$ with kernel $\k$.
            Suppose 
            \begin{equation*}
                \frac{1}{\k(0,0)} - \frac{1}{\k(z, w)} = q(z \bar w)   
            \end{equation*} 
            for a rational function $q$ with $q'(0) \ne 0$. If $\H$ does not extend to a larger disc, then $\H = H^2(\D)$ with equivalent norms.  
        \end{theorem}   
        The same result holds for $q$ analytic in a larger disc.
        \begin{theorem} \label{thrm: analytic 1 - 1/k}
            Let $\H$ be an irreducible complete Pick space of analytic functions on the unit disc $\D$ with kernel $\k$.
            Suppose 
            \begin{equation*}
                \frac{1}{\k(0,0)} - \frac{1}{\k(z, w)} = q(z \bar w)   
            \end{equation*} 
            for $q$ analytic in $R \D$ for some $R > 1$ with $q'(0) \ne 0$. If $\H$ does not extend to a larger disc, then $\H = H^2(\D)$ with equivalent norms.
        \end{theorem}
        As an example consider a space $\H$ with
        \begin{equation*}
            q(t) =\frac{t}{2(2 - t)}.
        \end{equation*}
        It gives us the Hardy space with a slightly modified norm, i.e.,
        \begin{equation*}
            \H = \left\{ f = \sum_{n = 0}^{\infty} a_n z^n: \: ||f||^2_{\H} = \frac{|a_0|^2}{2} + \sum_{n = 1}^{\infty} |a_n|^2 \right\},
        \end{equation*}
        but the embedding dimension of $\H$ is infinite.
        
        Next, we consider a family of Hardy-type spaces. 
        \begin{definition}
            \emph{The weighted Hardy space} $\H_s, \, s \in \R$ on $\D$ is 
            \begin{equation*}
                \H_s = \left\{ f(z) = \sum_{n \ge 0} a_n z^n: \: ||f||^{2}_{s} = \sum_{n \ge 0} (1 + n)^{-s} |a_n|^2 < \infty \right\}.
            \end{equation*}
        \end{definition}
        In particular, for $s = 0$ we get the regular Hardy space $H^2(\D)$ and for $s = -1$ we get the Dirichlet space.
        By \cite[Corollary 7.41]{AglBook}, $\H_s, \ s \le 0$ are complete Pick spaces.

        Similarly to \cite[Corollary 11.9]{hartz2017isomorphism} we find the embedding dimensions of these spaces.
        \begin{theorem} \label{thrm: embedding dimension of weighted}
            For $s = 0$, the embedding dimension of $\H_0$ is $d = 1$.
            For $s < 0$, the embedding dimension of $\H_s$ is $d = \infty$.
        \end{theorem}

        It is also interesting to consider the non-isometric version of the question, as it allows us to connect this problem to the isomorphism problem for analytic discs attached to the unit sphere. 

        \begin{theorem} \label{thrm: Hs isnt analytic disc}
            Suppose $V$ is an analytic disc attached to the unit sphere, $f$ is the embedding map of $V$. Then for $s < 0$
            \begin{equation*}
                \H_s \ne \H_f.
            \end{equation*}
        \end{theorem}

        Though, it remains unclear whether allowing $f$ which are not smooth up to the boundary is going to change the conclusion.

%% file: main/mainchap1.tex
\chapter{Analytic Discs} 
    \section{Properties of functions in \texorpdfstring{$\H_f$}{Hf} and \texorpdfstring{$\M_f$}{Mf}}
    We start by listing some necessary properties which must be satisfied by functions in the Hilbert spaces $\H_f$ and the multiplier algebras $\M_f$.
    \begin{theorem} \label{thrm: 1 = -1}
        Suppose $V$ is an analytic disc attached to the unit sphere, and $f$ is its embedding map such that $f(-1) = f(1)$. Let $h$ be a function in $\H_f$ such that the following limits exist
        \begin{equation*}
            h(1) = \lim_{r \to 1-} h(r), \qquad h(-1) = \lim_{r \to 1-} h(-r).
        \end{equation*}
        Then $h(1) = h(-1)$.
    \end{theorem}
    In particular, if $h \in \H_f$ extends continuously up to $\overline \D$, then $h(1) = h(-1)$.

    \begin{proof}
        By appropriately scaling $h$ we can assume $||h||_{\H_f} \le 1$. Since $\H_f$ is an RKHS with kernel $\k^f$ we can use \cite[Theorem 3.11]{paulsen2016introduction} to describe the functions in $\H_f$ in terms of the kernel:
        \begin{equation*}
            ||h||_{\H_f} \le 1 \iff \left( \k^f(z,w) - h(z) \overline{h(w)} \right)_{z, w \in \D} \ge 0.       
        \end{equation*}  
        In particular, using \eqref{eq: k^f formula}, for two points $z, w \in \D$ we obtain
        \begin{equation*}
            \begin{pmatrix}
                \frac{1}{1 - || f(z) ||^2} - |h(z)|^2 & \frac{1}{1 - \langle f(z), f(w) \rangle} - h(z) \overline{h(w)} \\
                \frac{1}{1 - \langle f(w), f(z) \rangle} - h(w) \overline{h(z)} & \frac{1}{1 - || f(w) ||^2} - |h(w)|^2
            \end{pmatrix}
            \ge 0.
        \end{equation*}
        As this condition implies that the determinant is positive, we have:
        \begin{align} 
            \left( \frac{1}{1 - || f(z) ||^2} - |h(z)|^2 \right) \left( \frac{1}{1 - || f(w) ||^2} - |h(w)|^2 \right) - \nonumber \\
            \label{eq: det>0}
            \left| \frac{1}{1 - \langle f(z), f(w) \rangle} - h(z) \overline{h(w)} \right|^2 \ge 0.
        \end{align}
    Now let us take $z = 1 - x, w = -1 + y$ for $x, y > 0$. Expanding $f$ around $1$ and $-1$ we get
    \begin{align*}
        f(1 - x) &= f(1) - f'(1)x + \frac{f''(1)}{2} x^2 + o(x^2), \\ 
        f(-1 + y) &= f(-1) + f'(-1)y + \frac{f''(-1)}{2} y^2 + o(y^2).
    \end{align*}
    Thus, 
    \begin{align} 
        & 1 - ||f(1 - x)||^2 = \nonumber \\ 
        & 2 \Re \langle f(1), f'(1) \rangle x - (||f'(1)||^2 + \Re \langle f(1), f''(1) \rangle)x^2 + o(x^2). \label{eq: 11.1}
    \end{align}
    Similarly,
    \begin{align} 
        & 1 - ||f(-1 + y)||^2 = \nonumber \\
        & -2 \Re \langle f(-1), f'(-1) \rangle y -  (||f'(-1)||^2 + \Re \langle f(-1), f''(-1) \rangle)y^2 + o(y^2). \label{eq: -1-1.1}
    \end{align}
    Note that by \cite[Proposition 3.1, Corollary 3.2]{Dav}, $A = \langle f(1), f'(1) \rangle > 0$ and $B = -\langle f(-1), f'(-1) \rangle > 0$. Let us set
    \begin{align*}
        C &= ||f'(1)||^2, \\
        2F &= \langle f(1), f''(1) \rangle, \\
        D &= ||f'(-1)||^2, \\
        2G &= \langle f(-1), f''(-1) \rangle.   
    \end{align*}
    We get
    \begin{align}
        1 - ||f(1 - x)||^2 &=  2Ax - (C + 2 \Re F)x^2 + o(x^2), \label{eq: 11.2} \\
        1 - ||f(-1 + y)||^2 &=  2By - (D + 2 \Re G)y^2 + o(y^2). \label{eq: -1-1.2}
    \end{align}
    Finally, noting that $f(1) = f(-1)$ and setting $E = \langle f'(1), f'(-1) \rangle$, we get
    \begin{equation} \label{eq: 1-1} 
        1 - \langle f(1-x), f(-1 + y) \rangle = Ax + By + E xy - \bar Fx^2 - Gy^2 + o(x^2) + o(y^2).  
    \end{equation}
    Expanding the brackets in \eqref{eq: det>0} we get
    \begin{align}
        \frac{1}{1 - || f(z) ||^2} \frac{1}{1 - || f(w) ||^2} - \left| \frac{1}{1 - \langle f(z), f(w) \rangle} \right|^2 \ge \label{eq: det:LHS} \\
        \frac{|h(w)|^2}{1 - || f(z) ||^2} + \frac{|h(z)|^2}{1 - || f(w) ||^2} - 2 \Re \left( \frac{\overline{h(z)} h(w)}{1 - \langle f(z), f(w) \rangle} \right) \label{eq: det:RHS}
    \end{align}
    To annihilate the highest order term in \eqref{eq: det:LHS} we set $Ax = By = t > 0$. Hence,
    \begin{align}
        \frac{1}{1 - || f(z) ||^2} = 
        \frac{1}{2t - \frac{C + 2 \Re F}{A^2}t^2 + o(t^2)} = \nonumber \\
        \frac{1}{2t} \left(1 + \frac{C + 2 \Re F}{2A^2}t + o(t) \right). \label{eq: 11.3}
    \end{align}
    Similarly,
    \begin{align}
        \frac{1}{1 - || f(w) ||^2} = 
        \frac{1}{2t - \frac{D + 2 \Re G}{B^2}t^2 + o(t^2)} = \nonumber \\
        \frac{1}{2t} \left( 1 + \frac{D + 2 \Re G}{2B^2}t + o(t) \right). \label{eq: -1-1.3}
    \end{align}
    Finally,
    \begin{align} 
        \frac{1}{1 - \langle f(z), f(w) \rangle} = 
        \frac{1}{2t + \left( \frac{E}{AB} - \frac{\bar F}{A^2} - \frac{G}{B^2} \right) t^2 + o(t^2)} = \nonumber \\
        \frac{1}{2t} \left(1 + \left( \frac{\bar F}{2A^2} + \frac{G}{2B^2} - \frac{E}{2AB} \right) t + o(t) \right). \label{eq: 1-1.2} 
    \end{align}
    Hence,
    \begin{align} 
        \left| \frac{1}{1 - \langle f(z), f(w) \rangle} \right|^2 =
        \frac{1}{4t^2} \left( 1 + \left( \frac{\Re F}{A^2} + \frac{\Re G}{B^2} - \frac{\Re E}{AB} \right) t + o(t) \right). \label{eq: 1-1.3}   
    \end{align}
    Substituting \eqref{eq: 11.3}, \eqref{eq: -1-1.3} and \eqref{eq: 1-1.3} into \eqref{eq: det:LHS} we get
    \begin{align}
        \frac{1}{4t^2} \Biggl( \left(1 + \frac{C + 2 \Re F}{2A^2}t + o(t) \right) \left( 1 + \frac{D + 2 \Re G}{2B^2}t + o(t) \right) - \nonumber \\
        \left( 1 + \left( \frac{\Re F}{A^2} + \frac{\Re G}{B^2} - \frac{\Re E}{AB} \right) t + o(t) \right) \Biggr) = \nonumber \\  
        \frac{1}{4t} \left( \frac{C}{2A^2} + \frac{\Re F}{A^2} + \frac{D}{2B^2} + \frac{\Re G}{B^2} - \frac{\Re F}{A^2} - \frac{\Re G}{B^2} + \frac{\Re E}{AB} + o(1) \right) = \nonumber \\
        \frac{1}{8t} \left( \frac{C}{A^2} + \frac{D}{B^2} + \frac{2 \Re E}{AB} \right) + o \left( \frac 1 t \right). \label{eq: det:LHS.2}
    \end{align}
    Similarly, expanding only up to $o(1)$ in \eqref{eq: 11.3}, \eqref{eq: -1-1.3}, \eqref{eq: 1-1.2} and substituting into \eqref{eq: det:RHS} we get 
    \begin{align}
        \frac{1}{2t} |h(w)|^2 (1 + o(1)) + \frac{1}{2t} |h(z)|^2 (1 + o(1)) - \frac{1}{2t} 2 \Re (\overline{h(z)} h(w)) (1 + o(1)) = \nonumber \\
        \frac{1}{2t} \left( |h(z)|^2 + |h(w)|^2 - 2 \Re (\overline{h(z)} h(w)) \right) + o \left( \frac 1 t \right) = \nonumber \\
        \frac{1}{2t} |h(z) - h(w)|^2 + o \left( \frac 1 t \right). \label{eq: det:RHS.2}
    \end{align}
    Thus, replacing \eqref{eq: det:LHS} with \eqref{eq: det:LHS.2} and \eqref{eq: det:RHS} with \eqref{eq: det:RHS.2} we get
    \begin{equation*}
        \frac{1}{8t} \left( \frac{C}{A^2} + \frac{D}{B^2} + \frac{2 \Re E}{AB} \right) + o \left( \frac 1 t \right) \ge \frac{1}{2t} |h(z) - h(w)|^2 + o \left( \frac 1 t \right).  
    \end{equation*}
    This is equivalent to
    \begin{equation*}
        |h(z) - h(w)|^2 \le \frac{1}{4} \left( \frac{C}{A^2} + \frac{D}{B^2} + \frac{2 \Re E}{AB} \right) + o(1).    
    \end{equation*}
    Explicitly writing $z = 1 - t/A$ and $w = - 1 + t/B$ we get 
    \begin{equation} \label{eq: k - k}
         \left|h \left( 1 - \frac{t}{A} \right) - h \left( -1 + \frac{t}{B} \right) \right|^2 \le \frac{1}{4} \left( \frac{C}{A^2} + \frac{D}{B^2} + \frac{2 \Re E}{AB} \right) + o(1), \quad t > 0.
    \end{equation}
    In particular, taking the limit as $t \to 0$ for $h$ that have radial limits at $\pm 1$ with $|| h ||_{\H_f} \le 1$ we get
    \begin{equation*}
        \left|h (1) - h \left( -1  \right) \right|^2 \le \frac{1}{4} \left( \frac{C}{A^2} + \frac{D}{B^2} + \frac{2 \Re E}{AB} \right).
    \end{equation*}
    It follows by scaling that in general we have
    \begin{equation*}
        \left|h (1) - h \left( -1  \right) \right|^2 \le \frac{1}{4} \left( \frac{C}{A^2} + \frac{D}{B^2} + \frac{2 \Re E}{AB} \right) || h ||^2_{\H_f},
    \end{equation*}
    for all $h \in \H_f$ with radial limits at $\pm 1$. Note that such functions are dense in $\H_f$, since, for example, all $\k^f_z, \, z \in \D$, \eqref{eq: k^f formula}, are continuous in $\overline \D$. 

    We conclude that there is a bounded functional on $\H_f$, let us denote the respective function by $\f \in \H_f$, such that $\langle h, \f \rangle_{\H_f} = h(1) - h(-1)$ for $h \in \H_f$ that have the limits 
    \begin{equation*}
        h(1) = \lim_{x \to 0+} h(1 - x), \qquad h(-1) = \lim_{y \to 0+} h(-1 + y).
    \end{equation*}
    To finish the proof of the theorem it remains to show that $\f = 0$. Indeed, for any $z \in \D$
    \begin{equation*}
        \f (z) = \langle \f, \k^f_z \rangle_{\H_f} = \overline{\langle \k^f_z, \f \rangle_{\H_f}} = \overline{\k^f_z(1) - \k^f_z(-1)} = 0.
    \end{equation*}
    \end{proof}

    \begin{corollary} \label{crl: k - k w 0}
        In the setting of Theorem \ref{thrm: 1 = -1} we have
        \begin{equation*}
            \k^f_{1 - \frac t A} - \k^f_{-1 + \frac t B} \to 0, \quad t \to 0.
        \end{equation*}
        in the $w^*$ topology.
    \end{corollary}
    \begin{proof}
        From \eqref{eq: k - k} it follows that 
        \begin{equation*}
            \k^f_{1 - \frac t A} - \k^f_{-1 + \frac t B}, \quad t \to 0.
        \end{equation*}
        is bounded in norm. Hence, since linear combinations of $\k^f_z, \ z \in \D$ are dense in $\H_f$, the conclusion follows from Theorem \ref{thrm: 1 = -1}. 
    \end{proof}

    An immediate corollary of Theorem \ref{thrm: 1 = -1} is the following
    \begin{theorem} \label{crl: xi = zet}
        Suppose $V$ is an analytic disc attached to the unit sphere, and $f$ is its embedding map. Let $f(\xi) = f(\z)$ for two distinct $\xi, \z \in \mathbb{T}$. If $h$ is a function in $\H_f$ that is continuous up to $\overline \D$, then $h(\xi) = h(\z)$.
    \end{theorem}
    \begin{proof}
        Consider a disc automorphism $\mu$ such that $\mu(-1) = \xi, \, \mu(1) = \z$. Then $g = f \circ \mu$ satisfies the conditions of Theorem \ref{thrm: 1 = -1}. Note that $h \circ \mu$ belongs to $\H_g$ ($\H_f \cong \H_g$ as RKHS by $\circ \mu$). $h \circ \mu$ is continuous up to $\overline \D$ as well, so it satisfies the conditions of Theorem \ref{thrm: 1 = -1}. We conclude $(h \circ \mu)(-1) = (h \circ \mu) (1)$, which is exactly $h(\xi) = h(\z)$.
    \end{proof} 

    From this result we can infer that the same must hold for multipliers.
    \begin{theorem} \label{thrm: crossings to multipliers}
        Suppose $V$ is an analytic disc attached to the unit sphere, and $f$ is its embedding map. Let $f(\xi) = f(\z)$ for two distinct $\xi, \z \in \mathbb{T}$. If $\f$ is a function in $\M_f$ that is continuous up to $\overline \D$, then $\f(\xi) = \f(\z)$.
    \end{theorem}
    \begin{proof}
        Consider $h(z) = \k_{0}^f(z)$. By definition $h(\xi) = h(\z) \ne 0$. Note that $\f h \in \H_f$ is continuous up to $\overline{\D}$, so that $\f(\xi) h(\xi) = \f(\z) h(\z)$. Dividing by $h(\xi) = h(\z) \ne 0$ we get $\f(\xi) = \f(\z)$.
    \end{proof}

    \section{Proof of Theorem \ref{thrm: algebras to crossings}}
    \begin{theorem*}
            Suppose $V$, $W$ are analytic discs attached to the unit sphere, and $f$, $g$ are the respective embedding maps. If $\M_f = \M_g$, then
            \begin{equation*}
                f(\xi) = f(\zeta) \iff g(\xi) = g(\zeta), \quad \xi, \zeta \in \TT. 
            \end{equation*}
            This means that $f$ and $g$ have the same self-crossings on the boundary.        
    \end{theorem*}

    \begin{proofof}[thrm: algebras to crossings]
        The result follows from Theorem \ref{thrm: crossings to multipliers}. Let us write $f$ as $f = (f_1, \ldots, f_d)$. Note that $z_j \in \M_d, \ j = 1, \ldots, d$, so that if we compose with $f$ we get $f_j \in \M_f, \ j = 1, \ldots, d$. Since $\M_f = \M_g$, by Theorem \ref{thrm: crossings to multipliers}, $g(\xi) = g(\z) \implies f_j(\xi) = f_j(\z), \ j = 1, \ldots, d$, which means exactly $f(\xi) = f(\z)$. Exchanging $f$, $g$ and repeating this argument we get the desired equivalence.    
    \end{proofof}

    \section{Proof of Theorem \ref{thrm: isomorphic algebras to automorphism}}
    First, we show that analytic discs are irreducible.

    \begin{lemma} \label{lem: irreducible varieties}
        If $V$ is an analytic disc attached to the unit sphere, then $V$ is irreducible in the sense of \cite[Section 5.1]{Sal}, i.e., for any regular point $\l \in V$, the intersection of zero sets of all multipliers vanishing on a small neighborhood $V \cap B_{\e}(\l)$ is exactly $V$.
    \end{lemma}
    \begin{proof}
        Suppose $V$ is not irreducible. This means that there is $\l \in V$ and $\e > 0$ such that the intersection of zero sets of all multipliers vanishing on $V \cap B_{\e}(\l)$ is a proper subset of $V$. Hence, there is $v \in V$ and $\f \in \M_d$ such that $\f|_{V \cap B_{\e}(\l)} = 0$ but $\f(v) \ne 0$. Let $f$ be the embedding map of $V$. Then, $v = f(a)$ for some $a \in \D$. Consider a function $\psi = \f \circ f$, it is analytic in $\D$. Since $\f|_{V \cap B_{\e}(\l)} = 0$, then $\psi|_{f^{-1}(B_{\e}(\l))} = 0$. But the set $f^{-1}(B_{\e}(\l))$ is open in $\D$, which means that $\psi = 0$. This contradicts the fact that $\psi(a) = \f(v) \ne 0$.
    \end{proof}
    
    Now, we can prove Theorem \ref{thrm: isomorphic algebras to automorphism}. 
    \begin{theorem*}
        Suppose $V$, $W$ are analytic discs attached to the unit sphere, and $f$, $g$ are the respective embedding maps. If $\M_f \cong \M_g$ algebraically, then there exists $\mu$, an automorphism of the unit disc $\D$, such that $\M_f = \M_{g \circ \mu}$ with equivalent norms.
    \end{theorem*}
    \begin{proofof}[thrm: isomorphic algebras to automorphism]
        Suppose that $\M_f \cong \M_g$ algebraically. 
        This is equivalent to $\M_V \cong \M_W$. Denote by $\Phi: \M_V \to \M_W$ the isomorphism.
        By Lemma \ref{lem: irreducible varieties}, the varieties $V$ and $W$ are irreducible. 
        Thus, by \cite[Theorem 5.5]{Sal}, there are maps $F, G: \B_d \to \B_d$ with multiplier coefficients, i.e., $F_j, G_j \in \M_d, \ j = 1, \ldots, d$, such that $F \circ G = \id_V$, $G \circ F = \id_W$ and the composition with these maps gives us $\Phi$, i.e.,
        \begin{equation*}
            \Phi(\f) = \f \circ F, \ \f \in \M_V, \qquad \Phi^{-1}(\psi) = \psi \circ G, \ \psi \in \M_W. 
        \end{equation*}
        To go back to the unit disc we define
        \begin{equation*}
            \mu = g^{-1} \circ G \circ f : \D \to \D,
        \end{equation*}
        so that
        \begin{equation*}
            \mu^{-1} = f^{-1} \circ F \circ g: \D \to \D.
        \end{equation*}
        Since $f^{-1}, g^{-1}$ are analytic on $V$, $W$ respectively, we see that $\mu$ is an analytic bijection from $\D$ onto itself, and, hence, a disc automorphism.
        
        By definition,
        \begin{equation*}
            \f \in \M_{g \circ \mu} \iff \f \circ \mu^{-1} \circ g^{-1} \in \M_W \iff \f \circ \mu^{-1} \circ g^{-1} \circ G \in \M_V \iff
        \end{equation*}
        \begin{equation*}
            \f \circ \mu^{-1} \circ g^{-1} \circ G \circ f \in \M_f \iff \f \in \M_f.
        \end{equation*}
        Since $\Phi$ is continuous as a homomorphism of semi-simple Banach algebras and the norms on $\M_f, \M_g$ are inherited from $\M_V, \M_W$ we conclude that $\M_f = \M_{g \circ \mu}$ with equivalent norms.    
    \end{proofof}

    \section{A semi-invariant for analytic discs}
        Define $A_f(\xi) = \langle f(\xi), f'(\xi) \xi \rangle, \ \xi \in \TT$. From \cite[Corollary 3.2]{Dav} we know that $A_f(\xi) > 0$ for embedding maps $f$. This function on the circle turns out to be an important semi-invariant for $\M_f$. First, we show how it changes under disc automorphisms.
        \begin{lemma} \label{lem: A_f under mu}
            Suppose $f$ is an embedding map, $\mu$ is a disc automorphism
            \begin{equation*}
                \mu(z) = \l \frac{\a - z}{1 - \bar \a z}, \quad \l \in \TT, \, \a \in \D.
            \end{equation*}
            Then
            \begin{equation*}
                A_{f \circ \mu}(\xi) = A_f (\mu(\xi)) \frac{1 - |\a|^2}{|\a - \xi|^2}. 
            \end{equation*}
        \end{lemma}
        \begin{proof}
            Expand
            \begin{align*}
                A_{f \circ \mu}(\xi) = \langle f(\mu(\xi)), (f \circ \mu)'(\xi) \xi \rangle = \langle f(\mu(\xi)), f'(\mu(\xi)) \mu'(\xi) \xi \rangle &= \\
                \langle f(\mu(\xi)), f'(\mu(\xi)) \mu(\xi) \frac{\mu'(\xi)}{\mu(\xi)} \xi \rangle &= \\
                \overline{\frac{\mu'(\xi)}{\mu(\xi)} \xi} A_f (\mu(\xi)).
            \end{align*}
            Hence, it is enough to show
            \begin{equation*}
                \frac{\mu'(\xi)}{\mu(\xi)} \xi = \frac{1 - |\a|^2}{|\a - \xi|^2}.  
            \end{equation*}
            Indeed,
            \begin{equation*}
                \mu'(z) = \l \frac{|\a|^2 - 1}{(1 - \bar \a z)^2},
            \end{equation*}
            so that
            \begin{equation*}
                \frac{\mu'(z)}{\mu(z)} = \frac{|\a|^2 - 1}{(\a - z)(1 - \bar \a z)}.
            \end{equation*}
            Thus,
            \begin{equation*}
                \frac{\mu'(\xi)}{\mu(\xi)} \xi = \frac{|\a|^2 - 1}{(\a - \xi)(1 - \bar \a \xi)} \xi = \frac{|\a|^2 - 1}{(\a - \xi)(\bar \xi - \bar \a)} = \frac{1 - |\a|^2}{|\a - \xi|^2}, 
            \end{equation*}
            which finishes the proof.
        \end{proof}
    
        \begin{theorem} \label{thrm: multipliers to bilip}
            Suppose $V$, $W$ are analytic discs attached to the unit sphere, and $f$, $g$ are the respective embedding maps. If $\M_f = \M_g$ with equivalent norms, then for $\xi, \z \in \TT$ such that $f(\xi) = f(\z)$ (and so Theorem \ref{thrm: algebras to crossings} implies $g(\xi) = g(\z)$) we have
            \begin{equation*}
                \frac{A_f(\xi)}{A_f(\z)} = \frac{A_g(\xi)}{A_g(\z)}.
            \end{equation*}
        \end{theorem}
        Let us introduce a notation we use in the proof. For two quantities $A,B > 0$ depending on some parameters, we write $A \asymp B$ whenever there exists a constant $C > 0$ that is independent of the parameters such that $A \le CB$ and $B \le CA$.
        \begin{proof}
            In light of Lemma \ref{lem: A_f under mu}, it is enough to prove this theorem for $\xi = 1, \z = -1$.

            Consider a metric on $\D$ induced by $\H_f$, see \cite[Lemma 9.9]{AglBook}:
            \begin{equation*}
                d_f(z,w) = \sqrt{1 - \frac{|\k_f(z,w)|^2}{\k_f(z,z)\k_f(w,w)}}.
            \end{equation*}
            It is easy to see that since $\H_f$ is complete Pick, we have
            \begin{equation*}
                d_f(z,w) = \sup \left\{ |\f(z)|: \: \f \in \M_f, \, ||\f||_{\M_f} \le 1, \, \f(w) = 0 \right\}.
            \end{equation*}
            We have a similar metric $d_g$ for $g$ instead of $f$. Since $\M_f = \M_g$ with equivalent norms, we conclude that the metrics $d_f$ and $d_g$ are equivalent, meaning that the identity map between $(\D, d_f)$ to $(\D, d_g)$ is bi-Lipschitz. Hence, $d_f^2$ and $d_g^2$ are equivalent as well, i.e.,
            \begin{equation*}
                1 - \frac{(1 - ||f(z)||^2)(1 - ||f(w)||^2)}{|1 - \langle f(z), f(w) \rangle|^2} \asymp 1 - \frac{(1 - ||g(z)||^2)(1 - ||g(w)||^2)}{|1 - \langle g(z), g(w) \rangle|^2}, \quad z, w \in \D.
            \end{equation*}
            Similarly to Theorem \ref{thrm: 1 = -1} we consider $z = 1 - x, w = -1 + y$ for $x, y > 0$ with $x, y \to 0$. Repeating the insight from Theorem \ref{thrm: 1 = -1} we set
            \begin{equation*}
                x = \frac{t}{A_f(1)}, \qquad y = \frac{t}{A_f(-1)},
            \end{equation*}
            for $t > 0$, $t \to 0$. This way, we have
            \begin{equation*}
                1 - ||f(z)||^2 = 2t + o(t), \quad 1 - ||f(w)||^2 = 2t + o(t), \quad 1 - \langle f(z), f(w) \rangle = 2t + o(t). 
            \end{equation*}
            Thus,
            \begin{equation*}
                \frac{(1 - ||f(z)||^2)(1 - ||f(w)||^2)}{|1 - \langle f(z), f(w) \rangle|^2} = \frac{(2t + o(t))(2t + o(t))}{|2t + o(t)|^2} = 1 + o(1), \quad t \to 0.    
            \end{equation*}
            We conclude that 
            \begin{equation*}
                d^2_f \left( 1 - \frac{t}{A_f(1)}, -1 + \frac{t}{A_f(-1)} \right) = o(1), \quad t \to 0. 
            \end{equation*} 
            Next, we look at what happens to $d_g^2$. Set $a = \frac{A_g(1)}{A_f(1)}$, $b = \frac{A_g(-1)}{A_f(-1)}$. We need to prove that $a = b$. We expand
            \begin{equation*}
                1 - ||g(z)||^2 = 2 a t + o(t), \quad 1 - ||g(w)||^2 = 2 b t + o(t), \quad 1 - \langle g(z), g(w) \rangle = \left( a + b \right) t + o(t).
            \end{equation*}
            Thus,  
            \begin{equation*}
                \frac{(1 - ||g(z)||^2)((1 - ||g(w)||^2))}{|1 - \langle g(z), g(w) \rangle|^2} = \frac{4ab}{(a + b)^2} + o(1), 
            \end{equation*}
            and so
            \begin{equation*}
                d^2_g \left( 1 - \frac{t}{A_f(1)}, -1 + \frac{t}{A_f(-1)} \right) = 1 - \frac{4ab}{(a + b)^2} + o(1), \quad t \to 0.
            \end{equation*}
            Since $d_g^2 \asymp d_f^2 = o(1), \ t \to 0$, we must have
            \begin{equation*}
                \frac{4ab}{(a + b)^2} = 1.
            \end{equation*}
            It remains to notice that
            \begin{equation*}
                \frac{4ab}{(a + b)^2} = 1 \iff (a + b)^2 = 4ab \iff (a - b)^2 = 0 \iff a = b.
            \end{equation*}
            We conclude that $a = b$, which means $\frac{A_f(1)}{A_g(1)} = \frac{A_f(-1)}{A_g(-1)}$, as we wanted.
        \end{proof}

    \section{Proof of Theorem \ref{thrm: multiplier isomorphism two points to two auto}}
    \begin{theorem*}
        Suppose $V$, $W$ are analytic discs attached to the unit sphere, and $f$, $g$ are the respective embedding maps with the only self-crossing at $\pm 1$. Then $\M_f \cong \M_g$ algebraically if and only if $\M_f = \M_{g \circ \mu}$, where
            \begin{equation*}
                \mu(z) = \frac{z - \a}{1 - \a z}, \text{ or} \qquad \mu(z) = \frac{\b - z}{1 - \b z}
            \end{equation*}
        with constants $\a$ and $\b$ that can be explicitly calculated in terms of $f(\pm 1), f'(\pm 1)$ and $ g(\pm 1), g'(\pm 1)$.
    \end{theorem*}
    \begin{proofof}[thrm: multiplier isomorphism two points to two auto]
        The fact that $\M_f = \M_{g \circ \mu}$ implies $\M_f \cong \M_g$ algebraically is evident for any automorphism $\mu$. Thus, we prove the other implication.

        Suppose $\M_f \cong \M_g$ algebraically. By Theorem \ref{thrm: isomorphic algebras to automorphism}, it means that there is a disc automorphism $\mu$ such that $\M_f = \M_{g \circ \mu}$ with equivalent norms. We claim that there are only two possible choices for $\mu$ as in the statement of the theorem.

        First, if $\M_f = \M_{g \circ \mu}$, then, by Theorem \ref{thrm: algebras to crossings}, $f$ and $g \circ \mu$ must have the same self-crossings on the boundary. But the only self-crossings for $f$ and $g$ are $\pm 1$, which means that $\mu(1) = \pm 1$ and $\mu(-1) = \mp 1$. We classify such $\mu$.

        \begin{lemma} \label{lem: -1 to -1 and 1 to 1}
            An automorphism $\mu$ satisfies $\mu(1) = 1$, $\mu(-1) = -1$ if and only if there is $\a \in (-1,1)$ such that
            \begin{equation*}
                \mu(z) = \frac{z - \a}{1 - \a z}.
            \end{equation*}
        \end{lemma}
        \begin{proof}
            The fact that $\mu$ of such form satisfies $\mu(1) = 1, \, \mu(-1) = -1$ is clear.

            Now suppose we know $\mu(1) = 1$ and $\mu(-1) = -1$. A general disc automorphism has the form
            \begin{equation*}
                \mu(z) = \l \frac{\a - z}{1 - \bar \a z}, \quad \l \in \TT, \a \in \D.
            \end{equation*}
            We have
            \begin{equation*}
                \l \frac{\a - 1}{1 - \bar \a} = 1, \qquad \l \frac{\a + 1}{1 + \bar \a} = -1.
            \end{equation*}
            It is equivalent to
            \begin{equation*}
                \l(\a - 1) = 1 - \bar \a, \qquad \l (\a + 1) = -1 - \bar \a.  
            \end{equation*}
            Subtracting one from another we get $\l = -1$ and substituting it back to the equations we get $\a = \bar \a$, meaning $\a \in (-1, 1)$. This finishes the proof.
        \end{proof}

        We go back to the proof of the theorem. Since $\mu(1) = \pm 1$ and $\mu(-1) = \mp 1$, either $\mu$ or $-\mu$ satisfies the requirements of Lemma \ref{lem: -1 to -1 and 1 to 1}. Hence, either
        \begin{equation*}
            \mu(z) = \frac{z - \a}{1 - \a z}, \quad \a \in (-1,1) 
        \end{equation*}
        or
        \begin{equation*}
            \mu(z) = \frac{\b - z}{1 - \b z}, \quad \b \in (-1,1).    
        \end{equation*}
        It remains to notice that $\a$ and $\b$ both have only one possible value. Indeed, if $\M_f = \M_{g \circ \mu}$, then in light of Theorem \ref{thrm: multipliers to bilip}
        \begin{equation} \label{eq: A/A = A/A}
            \frac{A_f(1)}{A_f(-1)} = \frac{A_{g \circ \mu}(1)}{A_{g \circ \mu}(-1)}.
        \end{equation}
        
        If we consider $\mu(z) = (z - \a)/(1 - \a z)$, then \eqref{eq: A/A = A/A} together with Lemma \ref{lem: A_f under mu} gives us
        \begin{equation*}
            \frac{A_f(1)}{A_f(-1)} = \frac{A_g(1)}{A_g(-1)} \left( \frac{\a + 1}{1 - \a} \right)^2,
        \end{equation*}
        so that
        \begin{equation} \label{eq: alpha}
            \a = \frac{\sqrt{A_f(1)A_g(-1)} - \sqrt{A_f(-1)A_g(1)}}{\sqrt{A_f(1)A_g(-1)} + \sqrt{A_f(-1)A_g(1)}}.
        \end{equation}
        
        Similarly, if we consider $\mu(z) = (\b - z)/(1 - \b z)$, then
        \begin{equation*}
            \frac{A_f(1)}{A_f(-1)} = \frac{A_g(-1)}{A_g(1)} \left( \frac{\b + 1}{1 - \b} \right)^2,            
        \end{equation*}
        so that
        \begin{equation} \label{eq: beta}
            \b = \frac{\sqrt{A_f(1)A_g(1)} - \sqrt{A_f(-1)A_g(-1)}}{\sqrt{A_f(1)A_g(1)} + \sqrt{A_f(-1)A_g(-1)}}.
        \end{equation}    
    \end{proofof}    

    \section{Proof of Theorem \ref{thrm: f_r and f_s}}
    \begin{theorem*}
        $\M_{f_r} \ne \M_{f_s}$ for $r \ne s$.
    \end{theorem*}
    \begin{proofof}[thrm: f_r and f_s]
        In light of Theorem \ref{thrm: multipliers to bilip} it is sufficient to prove that for $r > s > 0$
            \begin{equation*}
                \frac{A_{f_r}(1)}{A_{f_r}(-1)} \ne \frac{A_{f_s}(1)}{A_{f_s}(-1)}.
            \end{equation*}
            Evaluating at $\pm 1$ we get $f_r(1) = f_r(-1) = \frac{1}{\sqrt{2}}(1,1)$ and
            \begin{equation*}
                f_r'(1) = \frac{1}{\sqrt{2}} \left( 2, \, 2\frac{1 + r}{1 - r} \right), \qquad f_r'(-1) = -\frac{1}{\sqrt{2}} \left( 2, \, 2\frac{1 - r}{1 + r} \right).
            \end{equation*}
            Hence,
            \begin{equation*}
                \frac{A_{f_r}(1)}{A_{f_r}(-1)} = \frac{1 + \frac{1 + r}{1 - r}}{1 + \frac{1 - r}{1 + r}} = \frac{1 + r}{1 - r}
            \end{equation*}
            is strictly increasing for $r \in (0, 1)$, which finishes the proof.
    \end{proofof}
            
    \section{Proof of Theorem \ref{thrm: f_r-r and f_s-s}}
    \begin{theorem*}
        $\M_{f_{r, -r}} \cong \M_{f_{s, -s}}$ algebraically if and only if  $\M_{f_{r, -r}} = \M_{f_{s, -s}}$.
    \end{theorem*}  
    \begin{proofof}
        Note that 
            \begin{equation} \label{eq: f_r z to -z}
                f_{r, -r}(-z) = f_{-r, r}(z) = \frac{1}{\sqrt{2}}(b_{-r}(z)^2, b_r(z)^2),    
            \end{equation}
            so that if $f_{r, -r}'(1) = (z_1, z_2)$, then $f'_{r, -r}(-1) = -(z_2, z_1)$. 
            Since
            \begin{equation*}
                f_{r, -r}(1) = f_{r, -r}(-1) = \frac{1}{\sqrt{2}}(1, 1),    
            \end{equation*} 
            we get 
            \begin{equation} \label{eq: A(1) = A(-1)}
                A_{f_{r, -r}}(1) = A_{f_{r, -r}}(-1).     
            \end{equation}
        
            Let us apply Theorem \ref{thrm: multiplier isomorphism two points to two auto} to $f_{r, -r}$ and $f_{s, -s}$. Calculating $\a$, $\b$ from \eqref{eq: alpha}, \eqref{eq: beta} and using \eqref{eq: A(1) = A(-1)} we get $\a = \b = 0$. Thus, in light of \eqref{eq: f_r z to -z}, Theorem \ref{thrm: multiplier isomorphism two points to two auto} implies that $\M_{f_{r, -r}} \cong \M_{f_{s, -s}}$ algebraically if and only if either $\M_{f_{r, -r}} = \M_{f_{s, -s}}$ or $\M_{f_{r, -r}} = \M_{f_{-s, s}}$. 
            
            Note that because of \eqref{eq: f_r z to -z}
            \begin{equation*}
                \langle f_{r, -r}(z), f_{r, -r}(w) \rangle = \langle f_{-r, r}(z), f_{-r, r}(w) \rangle,
            \end{equation*}
            whence,
            \begin{equation*}
                \k^{f_{r, -r}} = \k^{f_{-r, r}},
            \end{equation*}
            meaning $\H_{f_{r, -r}} = \H_{f_{-r, r}}$ with the same inner product, which means $\M_{f_{r, -r}} = \M_{f_{-r, r}}$ isometrically.

            We conclude that $\M_{f_{r, -r}} \cong \M_{f_{s, -s}}$ algebraically if and only if $\M_{f_{r, -r}} = \M_{f_{s, -s}}$.    
    \end{proofof}
    
\chapter{Embedding dimension}
    \section{Preliminary results}  
        In this section we study properties of the rotation-invariant spaces of analytic functions on the unit disc.

        We start with a well-known result about kernel positivity, see \cite[Lemma 20]{mccarthy2017spaces}. 
        \begin{lemma} \label{lem: k positive c positive}
            Suppose $\k: \D \times \D \to \C$ is given by
            \begin{equation*}
                \k(z, w) = \sum_{n = 0}^{\infty} c_n (z \bar w)^n,
            \end{equation*}
            $c_n \in \C$ are such that $\k$ is well-defined, i.e., $\limsup_{n \to \infty} |c_n|^{\frac 1 n} \le 1$.
            Then $\k$ is positive semi-definite if and only if $c_n \ge 0, \ n \ge 0$.
        \end{lemma}
        For the sake of convenience, we present a short proof here. 
        \begin{proof}
            If $c_n \ge 0, \ n \ge 0$, then $\k$ is clearly positive semi-definite. We now prove the converse. By \cite[Theorem 2.53]{AglBook} there is a Hilbert space $\H$ and a function $g: \D \to \H$ such that
            \begin{equation*}
                \k(z, w) = \langle g(z), g(w) \rangle_{\H}.
            \end{equation*}
            Since $\k$ is analytic in $z$ and the image $g(w), \, w \in \D$ spans $\H$, $g$ is analytic. Hence, 
            \begin{equation*}
                g(z) = \sum_{n \ge 0} g_n z^n, \quad g_n \in \H,
            \end{equation*}
            so that
            \begin{equation*}
                \k(z, w) = \sum_{n, m \ge 0} \langle g_n, g_m \rangle_{\H} z^n \bar w^m.
            \end{equation*}
            In particular, $c_n = \langle g_n, g_n \rangle_{\H} = || g_n ||^2_{\H} \ge 0$.
        \end{proof}
        Now we can provide a complete description of rotation-invariant spaces. 
        \begin{theorem} \label{thrm: rotation inv to k}
            Let $\H$ be an RKHS of analytic functions on $\D$ with kernel $\k$. Then $\H$ is rotation-invariant if and only if the kernel $\k$ is given by
            \begin{equation} \label{eq: k = sum}
                \k(z, w) = \sum_{n \ge 0} c_n (z \bar w)^n, \quad c_n \ge 0.
            \end{equation}
            Moreover, if any (and, therefore, both) of these conditions are true, then for all $c_n \ne 0$ we have $z^n \in \H$, $|| z^n ||^2 = \frac{1}{c_n}$, and these monomials form an orthogonal basis in $\H$.
        \end{theorem}
        The first part of the theorem is the special case of \cite[Lemma 2.2]{hartz2017isomorphism} with the dimension being equal to $1$.
        \begin{proof}
            Note that $\H$ is rotation-invariant if and only if for any $\xi \in \TT$ we have 
            \begin{equation*}
                \k(\xi z, \xi w) = \k(z, w), \quad z,w \in \D.     
            \end{equation*}
            Hence, the sufficiency of \eqref{eq: k = sum} is clear. We now prove necessity.
            
            Since $\H$ is a space of analytic functions, $\k$ is analytic in the variables $z$ and $\bar w$ separately. Hence, by the Hartogs's theorem, $\k(z, \bar w)$ is analytic in $\D^2$. In particular,
            \begin{equation*}
                \k(z, w) = \sum_{n, m \ge 0} c_{nm} z^n \bar w^m,
            \end{equation*}
            for $z, w \in r \D$ with small enough $r > 0$. Since $\H$ is rotation-invariant, $\k(\xi z, \xi w) = \k(z, w)$, $\xi \in \TT$, so that
            \begin{equation*}
                \sum_{n, m \ge 0} c_{nm} z^n \bar w^m = \sum_{n, m \ge 0} c_{nm} \xi^{n - m} z^n \bar w^m, \quad z, w \in r \D.    
            \end{equation*}
            This implies $c_{nm} = c_{nm} \xi^{n - m}, \ n,m \ge 0$ for any $\xi \in \TT$. Hence, $c_{nm} = 0$ if $n \ne m$. Thus, setting $c_n = c_{nn}$, we get
            \begin{equation*}
                \k(z, w) = \sum_{n \ge 0} c_n (z \bar w)^n, \quad z,w \in r \D.
            \end{equation*}
            Since for a fixed $w \in \D$ the function $\k(z, w)$ is analytic in $\D$, this series extends to $z \in \D$. We conclude that
            \begin{equation*}
                \k(z, w) = \sum_{n \ge 0} c_n (z \bar w)^n, \quad z,w \in \D.
            \end{equation*}
            It remains to notice that since $\k$ is positive semi-definite, then, by Lemma \ref{lem: k positive c positive}, $c_n \ge 0, \ n \ge 0$. 
    
            By \cite[Theorem 3.11]{paulsen2016introduction}, a function $f: \D \to \C$ belongs to an RKHS $\H$ with $|| f ||_{\H} \le L$ if any only if
            \begin{equation*}
                \k(z,w) - \frac{f(z) \overline{f(w)}}{L^2}, \quad z, w \in \D.
            \end{equation*}
            is positive semi-definite.
    
            We consider $f(z) = z^n, \ n \ge 0$. If $c_n = 0$, then for any $L > 0$ the kernel
            \begin{equation*}
                \k(z, w) - \frac{(z \bar w)^n}{L^2}
            \end{equation*}
            is not positive semi-definite by Lemma \ref{lem: k positive c positive}, which means $z^n \notin \H$. If $c_n > 0$, then the kernel
            \begin{equation*}
                \k(z, w) - c_n (z \bar w)^n
            \end{equation*}
            is positive semi-definite by Lemma \ref{lem: k positive c positive}, so $z^n \in \H$ and $|| z^n ||^2 \le \frac{1}{c_n}$. But for $L^2 < \frac{1}{c_n}$ the kernel
            \begin{equation*}
                \k(z, w) - \frac{(z \bar w)^n}{L^2}
            \end{equation*}
            is not positive semi-definite by Lemma \ref{lem: k positive c positive}. We conclude that $|| z^n ||^2 = \frac{1}{c_n}$. 
            
            If $z^n, z^m \in \H, \ n \ne m$, then, by rotation-invariance, for any $\xi \in \TT$
            \begin{equation*}
                \langle z^n, z^m \rangle_{\H} = \langle (\xi z)^n, (\xi z)^m \rangle_{\H} = \xi^{n - m} \langle z^n, z^m \rangle_{\H}.    
            \end{equation*}
            Thus, $z^n \perp z^m$ in $\H$.
    
            Denote by $\H_{pol}$ the subspace of $\H$ spanned by $z^n$, $c_n \ne 0$. It remains to prove that the orthogonal complement of $\H_{pol}$ in $\H$ is $0$. Suppose $f \in \H_{pol}^{\perp}$. Note that for any $z \in \D$ $\k_z \in \H_{pol}$, so $f(z) = \langle f, \k_z \rangle = 0$, hence, $f = 0$. 
        \end{proof}
    
        Note that for a general RKHS $\H$ of analytic functions on $\D$ even if polynomials belong to $\H$, they are not necessarily dense in $\H$.
            As an example, consider the Hardy space $H^2(\D)$ and take any $f \notin H^2(\D)$, analytic in $\D$. We put $\H = H^2(\D) \oplus \C f$ with the direct sum inner product, i.e.,
            \begin{equation*}
                \langle g + af, h + bf \rangle_{\H} = \langle g, h \rangle_{H^2(\D)} + a \bar b, \quad g, h \in H^2(\D), \, a, b \in \C.
            \end{equation*}
            Then $\H$ is a Hilbert space of functions analytic in $\D$. Moreover, it is an RKHS with kernel
            \begin{equation*}
                \k^{\H}(z, w) = \frac{1}{1 - z \bar w} + f(z) \overline{f(w)}, \quad z, w \in \D.
            \end{equation*}
            It remains to notice that by construction $f \in \H$ is orthogonal to polynomials. Hence, polynomials are not dense in $\H$.

    \section{Proof of Theorem \ref{thrm: isometric embedding}}
        \begin{theorem*}
            Let $\H$ be a rotation-invariant irreducible complete Pick space of analytic functions on the unit disc $\D$ with kernel $\k$.
            Then the embedding dimension $d$ is finite if and only if 
            \begin{equation*}
                \frac{1}{\k(0, 0)} - \frac{1}{\k(z, w)}
            \end{equation*}
            is a polynomial in $z \bar w$. Moreover, $d$ is exactly the number of non-zero coefficients of this polynomial.
        \end{theorem*}
        \begin{proofof}[thrm: isometric embedding]
            We are looking for the smallest integer $d$ such that there is a
            map $b: \D \to \B_d$ and a rescaling function $\de: \D \to \C$ such that
            \begin{equation} \label{eq: kernel to drury-arveson kernel}
                \k(z, w) = \frac{\de(z) \overline{\de(w)}}{1 - \langle b(z), b(w) \rangle}.
            \end{equation}
            We can always compose $b$ with an automorphism of $\B_d$, which will give us a different embedding $b$ with a different rescaling function $\de$ but with the same dimension $d$. In particular, to determine the embedding dimension for a rotation-invariant $\H$ on $\D$ it is enough to consider only embeddings $b$ with $b(0) = 0$. 
    
            By Theorem \ref{thrm: rotation inv to k}, 
            \begin{equation*}
                \k(z, w) = \sum_{n \ge 0} c_n (z \bar w)^n,
            \end{equation*}
            where $c_n = \frac{1}{|| z^n ||^2_{\H}}$, $z^n \in \H$ and $c_n = 0$, $z^n \notin \H$. Since $\H$ is irreducible, $c_0 = \k(0, 0) \ne 0$. Hence, we can rescale $\H$ by a positive constant so that $c_0 = 1$. Note that this means $\k_0(z) = 1 = \k(0, 0)$.
            
            Putting $w = 0$ into \eqref{eq: kernel to drury-arveson kernel} together with $b(0) = 0$ we see that $\de(z) \overline{\de(0)} = 1$, which implies that $\de(z) = \xi$ for some $\xi \in \TT$. Substituting it back to \eqref{eq: kernel to drury-arveson kernel} we get
            \begin{equation} \label{eq: kernel to DA without delta}
                \k(z, w) = \frac{1}{1 - \langle b(z), b(w) \rangle}.   
            \end{equation}      
            
            Since $c_0 = 1$,
            \begin{equation*}
                \frac{1}{\k(z, w)} = 1 - \sum_{n \ge 1} r_n (z \bar w)^n.
            \end{equation*}
            We conclude that
            \begin{equation*}
                1 - \frac{1}{\k(z, w)} = \sum_{n \ge 1} r_n (z \bar w)^n.
            \end{equation*}
            From \eqref{eq: kernel to DA without delta} we get
            \begin{equation*}
                1 - \frac{1}{\k(z, w)} = \langle b(z), b(w) \rangle,
            \end{equation*}
            whence
            \begin{equation} \label{eq: 1 - 1/k = b}
                \langle b(z), b(w) \rangle = \sum_{n \ge 1} r_n (z \bar w)^n.   
            \end{equation}
            Since $b$ is analytic and $b(0) = 0$, 
            \begin{equation*}
                b(z) = \sum_{n \ge 1} b_n z^n,
            \end{equation*}
            where $b_n \in \B_d$, so that
            \begin{equation*}
                \langle b(z), b(w) \rangle = \sum_{n, m \ge 1} \langle b_n, b_m \rangle z^n \bar w^m.
            \end{equation*}
            It follows from \eqref{eq: 1 - 1/k = b} that the vectors $b_n$ are orthogonal and $||b_n||^2 = r_n$. 
            
            We conclude that if $1 - 1/\k$ is not a polynomial, i.e., infinitely many of the $r_n \ne 0$, then there are infinitely many non-zero orthogonal vectors $b_n \in \B_d$, so $d = \infty$. 
            
            If $1 - 1/\k$ is a polynomial and $r_{n_1}, \ldots, r_{n_k}$ are the non-zero coefficients, then $b_{n_1}, \ldots, b_{n_k}$ are non-zero and orthogonal, whence $d \ge k$. On the other hand, $d \le k$, since we can take        
            \begin{equation*}
                b(z) = \left(\sqrt{r_{n_1}} z^{n_1}, \ldots, \sqrt{r_{n_k}} z^{n_k} \right).   
            \end{equation*}    
        \end{proofof}
            
        
    \section{Proof of Theorem \ref{thrm: analytic 1 - 1/k}}
        \begin{theorem*}
            Let $\H$ be an irreducible complete Pick space of analytic functions on the unit disc $\D$ with kernel $\k$.
            Suppose 
            \begin{equation*}
                \frac{1}{\k(0,0)} - \frac{1}{\k(z, w)} = q(z \bar w)   
            \end{equation*} 
            for $q$ analytic in $R \D$ for some $R > 1$ with $q'(0) \ne 0$. If $\H$ does not extend to a larger disc, then $\H = H^2(\D)$ with equivalent norms.    
        \end{theorem*}
        \begin{proofof}[thrm: analytic 1 - 1/k]
            We rescale $\H$ by a positive constant so that $\k(0, 0) = 1$. The space $\H$ is clearly rotation-invariant,
            \begin{equation*}
                \k(z, w) = 1 + \sum_{n \ge 1} c_n (z \bar w)^n, \quad c_n \ge 0.
            \end{equation*}
            We need to prove that there are constants $0 < \e < M < \infty$ such that 
            \begin{equation*}
                \e < c_n < M, \quad n \ge 1.    
            \end{equation*}
            Since $\H$ is complete Pick,
            \begin{equation*}
                1 - \frac{1}{\k(z, w)} = \sum_{n \ge 1} r_n (z \bar w)^n, \quad r_n \ge 0, \ \sum_{n \ge 1} r_n \le 1.
            \end{equation*}
            Notice that
            \begin{equation*}
                \left( 1 - \sum_{m \ge 1} r_m z^m \right) \left( 1 + \sum_{n \ge 1} c_n z^n \right) = 1, \quad z \in \D.
            \end{equation*}
            Comparing the coefficients of $z^n, \, n \ge 1$ we get
            \begin{equation} \label{eq: r to c recurrence}
                c_n = \sum_{m = 0}^{n - 1} c_m r_{n - m}, \quad n \ge 1.  
            \end{equation}
            In particular, $c_n \ge c_{n - 1} r_1, \ n \ge 1$. Since $c_0 = 1$ and $r_1 > 0$ we conclude, by induction, that $c_n > 0, \ n \ge 0$. Similarly, applying induction, we get $c_n \le 1, \ n \ge 0$. Indeed, $c_0 = 1$, and if $c_0, \ldots, c_{n - 1} \le 1$, then
            \begin{equation*}
                c_n = \sum_{m = 0}^{n - 1} c_m r_{n - m} \le \sum_{m = 0}^{n - 1} r_{n - m} \le \sum_{m = 1}^{\infty} r_m \le 1.
            \end{equation*}
    
            It remains to show that $c_n$ are bounded from below for $n \to \infty$. Note that
            \begin{equation} \label{eq: mu = q'(1)}
                \mu = \sum_{n \ge 1} n r_n = q'(1) < \infty. 
            \end{equation}
            We also claim
            \begin{equation} \label{eq: 1 = q(1)}
                \sum_{n \ge 1} r_n = q(1) = 1.
            \end{equation}
            If not, then $q(1) < 1$, hence, $|q(z)| \le q(1) < 1, \, z \in \D$. Together with $q$ being analytic in $R \D$ we get that there is $\tilde R \in (1, R)$ such that $|q(z)| < 1, \, z \in \tilde R \D$. Thus,
            \begin{equation*}
                \k(z, w) = \frac{1}{1 - q(z \bar w)}
            \end{equation*}
            extends analytically to $\tilde R \D$, a contradiction with the fact that $\H$ does not extend to a larger disc.
    
            Together \eqref{eq: r to c recurrence}, \eqref{eq: mu = q'(1)}, \eqref{eq: 1 = q(1)} and $c_1 > 0$ imply, by the Erd\"os-Feller-Pollard theorem, \cite[Chapter XIII, Section 11]{Feller1968}, that $c_n \to 1/\mu, \ n \to \infty$. Since $1/\mu > 0$ and $c_n > 0, \, n \ge 0$ we conclude that there is $\e > 0$ such that $c_n > \e, \, n \ge 0$. This finishes the proof.
        \end{proofof}
            
        
    \section{Proof of Theorem \ref{thrm: rational q}}  
        \begin{theorem*}
            Let $\H$ be an irreducible complete Pick space of analytic functions on the unit disc $\D$ with kernel $\k$.
            Suppose 
            \begin{equation*}
                \frac{1}{\k(0,0)} - \frac{1}{\k(z, w)} = q(z \bar w)   
            \end{equation*} 
            for a rational function $q$ with $q'(0) \ne 0$. If $\H$ does not extend to a larger disc, then $\H = H^2(\D)$ with equivalent norms.  
        \end{theorem*}
        \begin{proofof}[thrm: rational q]
            We use the notation from the proof of Theorem \ref{thrm: analytic 1 - 1/k} and rescale so that $\k(0,0) = 1$.
        
            We have $q(1) = \sum_{n \ge 1} r_n \le 1$, so that $|q(z)| \le 1, \ z \in \D$. This means that the poles of $q$ lie outside $\overline{\D}$. Since there are finitely many of them, $q$ is analytic in a larger disc and we can apply Theorem \ref{thrm: analytic 1 - 1/k}.
        \end{proofof}
    
    \section{Proof of Theorem \ref{thrm: embedding dimension of weighted}}
        Recall that the weighted Hardy space $\H_s, \, s \in \R$ on $\D$ is 
        \begin{equation*}
            \H_s = \left\{ f(z) = \sum_{n \ge 0} a_n z^n: \: ||f||^{2}_{s} = \sum_{n \ge 0} (1 + n)^{-s} |a_n|^2 < \infty \right\}.
        \end{equation*}
        \begin{theorem*}
            For $s = 0$, the embedding dimension of $\H_0$ is $d = 1$.
            For $s < 0$, the embedding dimension of $\H_s$ is $d = \infty$.
        \end{theorem*}
        \begin{proofof}[thrm: embedding dimension of weighted]
            The statement for $s = 0$ is evident, since the space $\H_0$ is exactly the Drury-Arveson space $H^2_d$ with $d = 1$, meaning that $b(z) = z$ is the map in question.  
            
            Let us now fix $s < 0$. Note that the kernel is given by
            \begin{equation*} \label{eq: weighted Hardy kernel}
                \k(z, w) = \sum_{n \ge 0} (1 + n)^{s} (z \bar w)^n, \quad z,w \in \D.
            \end{equation*}
            If $d < \infty$, by Theorem \ref{thrm: isometric embedding}, we must have
            \begin{equation*}
                1 - \frac{1}{\k(z, w)} = p(z \bar w), 
            \end{equation*}
            for some polynomial $p$, $p(0) = 0$. Note that $p'(0) \ne 0$. 
            But then, by Theorem \ref{thrm: rational q}, we must have $\H_s = H^2(\D)$ with equivalent norms, which contradicts the fact that $||z^n||_{\H_s} = (1 + n)^{-s} \to \infty$.    
        \end{proofof}

    \section{Proof of Theorem \ref{thrm: Hs isnt analytic disc}}
        \begin{theorem*}
            Suppose $V$ is an analytic disc attached to the unit sphere, $f$ is the embedding map of $V$. Then for $s < 0$
            \begin{equation*}
                \H_s \ne \H_f.
            \end{equation*}    
        \end{theorem*}
        \begin{proofof}[thrm: Hs isnt analytic disc]
            Suppose the opposite is true, i.e., $\H_s = \H_f$. 
            First, by Theorem \ref{crl: xi = zet}, $f$ has to be injective, since $z \in \H_s$ is injective and continuous up to $\overline{\D}$. For injective $f$, by Theorem \ref{thrmm: injective embedding}, $\H_f = H^2(\D)$ while $\H_s \ne H^2(\D)$, which gives a contradiction.     
        \end{proofof}
         

%% file: main/conclusion.tex
\chapter{Conclusion}

This thesis has delved into the isomorphism problem for multiplier algebras associated with varieties in the unit ball, focusing specifically on the case of analytic discs attached to the unit sphere. 
Our results have provided new insights into the relationship between the geometric properties of these varieties, such as their self-crossing types, and the algebraic structure of their associated multiplier algebras.

We established (Theorem \ref{thrm: algebras to self-crossings}) that for analytic discs with different self-crossing types, the corresponding multiplier algebras are non-isomorphic, thus offering a coarse criterion to distinguish between non-isomorphic cases.
However, the question of whether the same self-crossing type necessarily implies isomorphism of the multiplier algebras remains open and warrants further investigation.
Additionally, we have proven a rigidity result (Theorem \ref{thrm: multiplier isomorphism two points to two auto}) for analytic discs with a single self-crossing, which states that, in this case, there are always at most two candidates for an isomorphism map.

We also examined the embedding dimension problem for complete Pick spaces. 
We considered rotation-invariant spaces on the unit disc and established a general result (Theorem \ref{thrm: isometric embedding}) that explicitly relates the embedding dimension with the kernel of the space. 
We showed (Theorem \ref{thrm: rational q}, Theorem \ref{thrm: analytic 1 - 1/k}) that under mild conditions on the embedding map the resulting rotation-invariant space on the unit disc is, in fact, the Hardy space with an equivalent norm.
By using these two results, we were able to demonstrate (Theorem \ref{thrm: embedding dimension of weighted}) that the embedding dimension for certain weighted Hardy-type spaces is infinite. 
Lastly, we made a connection between the embedding dimension and the isomorphism problem. We showed (Theorem \ref{thrm: Hs isnt analytic disc}) that weighted Hardy-type spaces cannot arise from analytic discs attached to the unit sphere.
